\documentclass[12pt]{article}

\usepackage{amssymb,amsmath, amsthm, amscd,ifthen}
\usepackage[T2A]{fontenc}
\usepackage[utf8]{inputenc}
\usepackage[english]{babel}
\usepackage{hyperref}
\usepackage{indentfirst}
\usepackage{euscript}
\usepackage{tikz}
\usetikzlibrary{calc,intersections,through,backgrounds,arrows}

\newcounter{ris}

\newtheorem{theorem}{Theorem}[section]
\newtheorem{lemma}[theorem]{Lemma}

\newtheorem{corollary}[theorem]{Corollary}

\theoremstyle{definition}
\newtheorem{question}[theorem]{Question}
\newtheorem{remark}[theorem]{Remark}
\newtheorem{definition}[theorem]{Definition}
\newtheorem{conjecture}[theorem]{Conjecture}
\newtheorem{example}[theorem]{Example}
\newtheorem{notation}[theorem]{Notation}

\def\p{\EuScript P}
\def\q{\EuScript Q}
\def\r{\EuScript R}
\def\s{\EuScript S}
\def\x{\EuScript X}
\def\y{\EuScript Y}

\begin{document}

\title{Decompositions of functions defined on finite sets in $\mathbb{R}^d$}
\author{
 Khaydar Nurligareev\footnote{LIPN, University Sorbonne Paris Nord},
 Ivan Reshetnikov\footnote{Moscow Institute of Physics and Technology}
}
%\date{\today}

\maketitle

\begin{abstract}
 A finite subset $M \subset \mathbb{R}^d$ is \emph{basic}, if for any function $f \colon M \to \mathbb{R}$ there exists a collection of functions $f_1, \ldots, f_d \colon \mathbb{R} \to \mathbb{R}$ such that for each element $(x_1, \ldots, x_d)\in M$ we have $f(x_1, \ldots, x_d) = f_1(x_1) + \ldots + f_d(x_d)$. For certain finite sets, we prove a~criterion for a~set to be basic, and we show that it cannot be extended to the general case. In addition, we interpret the above criterion in terms of doubly-weighted graphs and give an estimation for the number of elements in certain basic and non-basic subsets.

\end{abstract}

\section{Introduction}

 The concept of \textit{basic subsets} arises in connection with Hilbert's thirteenth problem on the superposition of continuous functions. The first time it was introduced in an explicit form by Sternfeld in 1989 \cite{Sternfeld1989}. He called a~subset $M \subset \mathbb{R}^d$ \emph{(continuously) basic}, if for any continuous function $f \colon M \to \mathbb{R}$ there exists a collection of continuous functions $f_1, \ldots, f_d \colon \mathbb{R} \to \mathbb{R}$ such that
 \begin{equation}\label{eq: basic set condition}
  f(x_1, \ldots, x_d) = f_1(x_1) + \ldots + f_d(x_d)
 \end{equation}
 for each element $(x_1, \ldots, x_d)\in M$. The origin of this concept goes back to 1958~\cite{Arnold1958} when Arnold raised the question equivalent to the following: what are basic subsets in the case $d=2$? Sternfeld showed that a closed bounded subset $M\subset\mathbb{R}^2$ is basic if and only if $M$ does not contain arbitrary long arrays, where an \emph{array} is a (finite or infinite) sequence of points $(x_i,y_i)$ on the plane such that $x_i=x_{i+1}$, $y_i\ne y_{i+1}$ for odd $i$ and $y_i=y_{i+1}$, $x_i\ne x_{i+1}$ for even $i$.

 In this paper, we focus on finite subsets of $\mathbb{R}^d$, where $d\geqslant2$. In this case, the condition of continuity can be omitted. Without loss of generality, we identify finite subsets of~$\mathbb{R}^d$ with integer points inside $d$-dimensional cube~$[n]^d$, where $[n] = \{1,\ldots,n\}$. We establish the following estimation for the number of elements in basic subsets.

\begin{theorem}\label{th: M is basic => |M| < dn - (d-2)}
  If $M \subset [n]^d$ is a basic subset, then
  $$
    |M| \leqslant dn - (d-1).
  $$
  This boundary cannot be improved.
\end{theorem}

 On the other hand, we estimate the number of elements in non-basic subsets that are minimal in the sense of inclusion (to be precise, we call a~non-basic subset $M\subset[n]^d$ \emph{minimal}, if any proper subset $K\subset M$ is basic). In order to obtain non-trivial results, we suppose that each layer of $[n]^d$ contains an element of a~subset, where by \emph{layers} we mean hyperplanes orthogonal to the coordinate axes, so that each layer can be determined by the equation $x_i=j$, where $i\in[d]$ and $j\in[n]$.

\begin{theorem}\label{th: M is non-basic => 2n-1 < |M| < dn - (d-3)}
  If $M \subset [n]^d$ is a minimal non-basic subset such that every layer of~$[n]^d$ has a non-empty intersection with $M$, then
  $$
    2n \leqslant |M| \leqslant dn - (d-2).
  $$
\end{theorem} 

 The condition for a subset $M \subset [n]^d$ to be basic is equivalent to the consistency of the corresponding system of linear equations with integer coefficients. In particular, there exists an algorithm, polynomial in $|M|$, for determining whether the subset $M$ is basic. For $d=2$, however, one can indicate a~simpler algorithm based on the following criterion: a finite subset $M \subset [n]^2$ is basic if and only if $M$ does not contain a \emph{closed array}, that is, a~non-trivial finite array with the coinciding first and last points \cite{Skopenkov2010}. The reader can think about this criterion in the following way. Let us color the points of a closed array in two colors, say, color odd points in red and even points in blue respectively. Then every layer contains the same number of red and blue points. Extending the idea of coloring onto $\mathbb{R}^d$, we get the following theorem. 
 
\begin{theorem} \label{th_set_to_graph}
 Let $M\subset[n]^d$ be a subset containing two or zero elements in every layer. Then $M$ is non-basic if and only if there exists a non-empty subset $K\subset M$ and a coloring of $K$ in two colors such that every layer contains the same number of elements of each color.
\end{theorem}

 It turns out that the same ideas can be used for studying doubly-weighted graphs as well. In particular, we establish the following result which is of independent interest.

\begin{theorem} \label{th_basis_graph}
 A graph $G = (V, E)$ does not contain a bipartite connected component if and only if for any vertex weight function $w_V \colon V \to \mathbb{R}$ there exists an edge weight function $w_E \colon E \to \mathbb{R}$ such that the weight of any vertex is equal to the sum of weights of the edges incident to this vertex:
 \[
  w_V(v) = \sum\limits_{e\colon v\in e} w_E(e).
 \]
\end{theorem}

 Assigning colors to elements of a finite subset $M\subset[n]^d$ is a particular case of \emph{weight functions} $M\to\mathbb{Z}$. Identifying red and blue colors with values of the set $\{-1,1\}$, one can notice that if a non-basic subset $M$ satisfies Theorem~\ref{th_set_to_graph}, then, for a subset $K\subset M$, the sum of the weight function values taken over a fixed layer is zero. This observation leads us to the following definition. We will call $f\colon M\to\mathbb{Z}$ \emph{annihilation function of~$M$}, if for each layer $L$:
 \[
  \sum\limits_{\x\in L\cap M}f(\x) = 0.
 \]
 It turns out that basic and minimal non-basic subsets admit the following descriptions in terms of their annihilation functions.

\begin{lemma}\label{lemma: M is non-basic <=> annihilation weight function}
 A subset $M \subset [n]^d$ is non-basic if and only if there exists a~non-trivial annihilation function of $M$.
\end{lemma}

\begin{lemma}\label{lemma: M is min. non-basic <=> annihilation weight function is unique}
 If a non-basic subset $M \subset [n]^d$ is minimal, then the annihilation function is unique up to multiplying by a constant.
\end{lemma}
 
 As we have just seen, when a non-basic subset $M$ satisfies conditions of Theorem~\ref{th_set_to_graph}, one can choose an annihilation function of $M$ whose values are in $\{-1,0,1\}$. In general, however, the values are not bounded. More precisely, consider the annihilation function of a minimal non-basic subset~$M$ whose values are setwise coprime integers (we will call such annihilation functions \emph{irreducible}). Then the following result takes place.

\begin{theorem}\label{th_irreducible_annihilation_function_is_unbounded}
 For any positive integer $m$ there exist a positive integer $n$ and a minimal non-basic subset $M\subset[n]^3$ whose irreducible annihilation function~$f$ has the value $f(\x)=m$ for some $\x \in M$.
\end{theorem}

 Since the coefficients of irreducible annihilation functions are unbounded, one cannot expect to have a criterion similar to Theorem~\ref{th_set_to_graph} in the general case. Nevertheless, one can obtain a simplification by considering an annihilation function $f$ of a subset $M\subset[n]^d$ as a function $f\colon [n]^d \to \mathbb{Z}$ such that $f(\x)=0$ as long as $\x\notin M$. The simplest non-trivial annihilation functions of $[n]^d$ correspond to the simplest closed arrays, i.e. rectangles (we call such annihilation functions \emph{simple}, see Definition~\ref{def: simple annihilation function}). It turns out that any annihilation function can be generated as a sum of simple annihilation functions. This observation gives us a rather simple way to construct non-trivial non-basic subsets of $[n]^d$ as domains of annihilation functions of $[n]^d$. 

\begin{theorem} \label{thBoyarov}
 Every annihilation function of $[n]^d$ can be decomposed into a~finite sum of simple annihilation functions.
\end{theorem}

 The structure of the paper is the following. In Section~\ref{Section:algebra_app}, we translate the concept of finite basic subsets into the language of systems of linear equations and study their properties. Section~\ref{Section:estimations}, Section~\ref{Section:graphs_app} and Section~\ref{Section:structure_of_non-basic_sets} are devoted to the proofs of Theorems~\ref{th: M is basic => |M| < dn - (d-2)} and~\ref{th: M is non-basic => 2n-1 < |M| < dn - (d-3)}, Theorems~\ref{th_set_to_graph} and~\ref{th_basis_graph}, Theorems~\ref{th_irreducible_annihilation_function_is_unbounded} and~\ref{thBoyarov} respectively. We end the paper by Section~\ref{Section:conclusion} discussing possible directions for further research.

\section{Finite basic subsets and systems of linear equations} \label{Section:algebra_app}

% Recall that by $[n]$ we mean the set $\{1,\ldots,n\}$.

%\begin{definition}\label{defin: basic subset}
% A~subset $M \subset [n]^d$ is \emph{basic}, if for any function $f \colon M \to \mathbb{R}$ there exists a collection of functions $f_1, \ldots, f_d \colon \mathbb{R} \to \mathbb{R}$ such that
% \begin{equation}\label{eq: basic set condition}
%  f(x_1, \ldots, x_d) = f_1(x_1) + \ldots + f_d(x_d)
% \end{equation}
% for each point $(x_1, \ldots, x_d)\in M$.
%\end{definition}

%\begin{definition}\label{def: layer}
% The hyperplanes defined by equations
% $$
%  x_i=j,
% $$
% where $i\in[d]$ and $j\in[n]$, will be called \emph{layers}. 
%\end{definition}

  Let $M$ be a subset of $[n]^d$. Then, in algebraic terms, condition~\eqref{eq: basic set condition} for~$M$ to be basic corresponds to the consistency of the system of $|M|$~linear equations in $dn$~variables. Indeed, for every $i \in [d]$, the variable $x_i$ generates $n$~different layers $x_i = 1$, $\ldots$, $x_i = n$. For every $j \in [n]$, j-th layer corresponds to the unique value~$f_i(j)$ that can be interpreted as a~variable~$X_{ij}$. Thus, \eqref{eq: basic set condition} is equivalent to the system of linear equations of the form
  \begin{equation}\label{eq: basic set condition, algebraic form}
    f(\x) = f(x_1,\ldots,x_d) = \sum\limits_{i=1}^d X_{ix_i},
  \end{equation}
  where $\x = (x_1,\ldots,x_d) \in M$. In particular, if this system is consistent for any function $f\colon M\to\mathbb{R}$, then $M$ is basic, and vice versa.

\begin{notation}\label{notation: A_M}
  We denote $A_M$ the matrix of system~\eqref{eq: basic set condition, algebraic form}.
\end{notation}

\begin{remark}
  The set of all functions $f\colon M \to \mathbb{R}$ form a vector space of dimension~$|M|$ with the basis of \emph{indicator functions} $\textbf{1}_{\x}$,
  \[
   \textbf{1}_{\x}(\y) = \left\{\begin{array}{cl}
    1, & \y = \x \\
    0, & \y \ne \x,
   \end{array}\right.
  \]
  where $\x,\y \in M$. Indeed, every function $f$ can be represented as a linear combination of the indicator functions:
  \[
    f = \sum\limits_{\x \in M}f(\x)\textbf{1}_{\x}.
  \]
\end{remark}

%\begin{definition}\label{defin: annihilation function}
% Given a finite set $M\subset[n]^d$, we call $f\colon M\to\mathbb{Z}$ \emph{an annihilation function of $M$}, if for each layer $L$:
% \[
%  \sum\limits_{\x\in L\cap M}f(\x) = 0.
% \]
%\end{definition}

\begin{proof}[Proof of Lemma~\ref{lemma: M is non-basic <=> annihilation weight function}]
  A subset $M \subset [n]^d$ is non-basic if and only if the rows of the matrix~$A_M$ of system~\eqref{eq: basic set condition, algebraic form} are linearly dependent. In other words, there exists a non-trivial linear combination of rows which equals zero. Since the rows and columns of~$A_M$ correspond to the elements of $M$ and the layers respectively, coefficients of this linear combination are the values of annihilation function. Note, that the entries of~$A_M$ are zeroes and ones, therefore, it is possible to choose a~linear combination with integer coefficients. 
\end{proof}

%\begin{definition}\label{defin: minimal non-basic set}
% We call a non-basic set $M\subset[n]^d$ \emph{minimal}, if any proper subset $K\subset M$ is basic.
%\end{definition}

\begin{proof}[Proof of Lemma~\ref{lemma: M is min. non-basic <=> annihilation weight function is unique}]
  The existence of two different annihilation functions corresponds to the existence of two different linear combinations of rows of~$A_M$ that equal zero. The latter implies that we can construct a non-trivial linear combination of rows such that at least one of its coefficients is zero. Hence, there is a proper subset $K\subset M$ which is non-basic.
\end{proof}

\begin{remark}
  The converse of Lemma~\ref{lemma: M is min. non-basic <=> annihilation weight function is unique} does not hold. For example, let
$$
 M=\{(1,1,1),(2,1,1),(1,2,1),(1,1,2),(2,2,1)\}.
$$
  Then $M$ is not minimal since we can eliminate $(1,1,2)$ and obtain a~closed plane array (see Fig.~\ref{figure: non-minimal non-basic set}). Nevertheless, any annihilation function of $M$ is a~product of
$$
 f = \textbf{1}_{(1,1,1)} - \textbf{1}_{(2,1,1)} + \textbf{1}_{(2,2,1)} - \textbf{1}_{(1,2,1)}
$$
and some constant.

 We can see that in this example $f(1,1,2) = 0$. In fact, the uniqueness of the annihilation function implies minimality of a non-basic subset $M$, if we additionally require that $f(\x)\ne0$ for any $\x\in M$.
\end{remark}

\begin{center}
\begin{tikzpicture}[scale=1.2,x={(1cm,0cm)},y={(0.54cm,0.36cm)},z={(0cm,1cm)},line width=.5pt]
 \begin{scope}
  \coordinate (a1) at (0,0,0);
  \coordinate (a2) at (1,0,0);
  \coordinate (a3) at (1,1,0);
  \coordinate (a4) at (0,1,0);
  \coordinate (b1) at (0,0,1);
  \coordinate (b2) at (1,0,1);
  \coordinate (b3) at (1,1,1);
  \coordinate (b4) at (0,1,1);
  \draw [line width=.8pt] (b1)--(b4)--(b3)--(b2)--(b1)--(a1)--(a2)--(b2);
  \draw [line width=.8pt] (a2)--(a3)--(b3);
  \draw [dashed] (a1)--(a4);
  \draw [dashed] (b4)--(a4)--(a3);
  \foreach \p in {b2,b3,b4}{
    \filldraw[white] (\p) circle (1.7pt);
    \draw (\p) circle (1.7pt);
  }
  \foreach \p in {a1,a2,a3,a4,b1}
    \filldraw (\p) circle (1.7pt);
  \refstepcounter{ris}
  \draw (0.5,0.5,-0.6) node {Figure \arabic{ris}.\label{figure: non-minimal non-basic set}};
 \end{scope}
\end{tikzpicture}
\end{center}

%[Sufficient condition for a basic set] 
\begin{lemma}\label{lemma: one point in layer => basic}
 A subset $M \subset [n]^d$ is basic, if for any non-empty subset $K \subset M$ there is a layer containing only one element from $K$.
\end{lemma}
\begin{proof}
  In terms of matrix $A_M$, the conditions of Lemma~\ref{lemma: one point in layer => basic} mean that any collection of rows of $A_M$ generates a submatrix having a column with the unique non-zero entry. Hence, there is no non-trivial linear combination of rows of $A_M$ which equals zero. This implies that $M$ is basic.
\end{proof}

\section{Estimations for basic and non-basic subsets} \label{Section:estimations}

  The main goal of this section is to prove Theorems~\ref{th: M is basic => |M| < dn - (d-2)} and~\ref{th: M is non-basic => 2n-1 < |M| < dn - (d-3)}.

\begin{lemma}\label{lemma: estimations}
  Let $n\in\mathbb{N}$ and $Y_{ij}$ be coordinates in $\mathbb{R}^{dn}$, where $i\in[d]$ and $j\in[n]$. If the entries of each of $m$ vectors from $\mathbb{R}^{dn}$ satisfy $(d-1)$ equations
  \begin{equation}\label{eq: subspace condition}
    Y_{11} + \ldots + Y_{1n} = \ldots = Y_{d1} + \ldots + Y_{dn}
  \end{equation}
  and $m > dn-(d-1)$, then these $m$ vectors are linearly dependent.
\end{lemma}
\begin{proof}
  Each vector is contained in the subspace $V \subset\mathbb{R}^{dn}$ determined by condition~\eqref{eq: subspace condition}. Hence, $\dim V = dn - (d-1) < m$, which implies that any $m$~vectors in $V$ are linearly dependent.
\end{proof}

\begin{proof}[Proof of Theorem~\ref{th: M is basic => |M| < dn - (d-2)}]
  Let us suppose that $|M| > dn - (d-1)$. Then the rows of the matrix~$A_M$ satisfy conditions of Lemma~\ref{lemma: estimations} for $m=|M|$. Hence, they are linearly dependent, meaning, that $M$ is not basic.

  To show that it is impossible to improve the inequality, it is sufficient to present an appropriate example. For this purpose, set $M$ to be as follows:
  $$
   M = \{(k,1,\ldots,1), (1,k,\ldots,1), \ldots, (1,1,\ldots,k) \mid k\in[n]\}
  $$
  (Figure~\ref{figure: maximal basic set in 4x4x4} illustrates the set $M$ in the case $d=3$, $n=4$). Then $M$ is basic by Lemma~\ref{lemma: one point in layer => basic} and $|M| = dn - (d-1)$.
\end{proof}

\begin{center}
\begin{tikzpicture}[scale=1.2,x={(1cm,0cm)},y={(0.3cm,0.2cm)},z={(0cm,1cm)},line width=.5pt]
 \begin{scope}
  \coordinate (a) at (0,0,0);
  \coordinate (b) at (1,0,0);
  \coordinate (c) at (2,0,0);
  \coordinate (d) at (3,0,0);
  \coordinate (e) at (0,1,0);
  \coordinate (f) at (0,2,0);
  \coordinate (g) at (0,3,0);
  \coordinate (h) at (0,0,1);
  \coordinate (i) at (0,0,2);
  \coordinate (j) at (0,0,3);
  \foreach \p in {0,1,2,3}{
    \draw [line width=.8pt] (0,\p,3)--++(3,0,0)--++(0,0,-3);
    \draw [line width=.8pt] (0,0,\p)--++(3,0,0)--++(0,3,0);
  }
  \foreach \p in {0,1,2}{
    \draw [line width=.8pt] (\p,0,0)--++(0,0,3)--++(0,3,0);
    \draw [dashed] (0,\p+1,3)--++(0,0,-3)--++(3,0,0);
    \draw [dashed] (\p,0,0)--++(0,3,0)--++(0,0,3);
  }
  \foreach \p in {1,2}{
   \draw [dashed] (0,0,\p)--++(0,3,0)--++(3,0,0);
   \foreach \q in {1,2}{
    \draw [dashed] (0,\p,\q)--++(3,0,0);
    \draw [dashed] (\p,0,\q)--++(0,3,0);
    \draw [dashed] (\p,\q,0)--++(0,0,3);
   }
  }
  \foreach \x in {0,1,2,3}{
   \foreach \y in {0,1,2,3}{
    \foreach \z in {0,1,2,3}{
     \filldraw[white] (\x,\y,\z) circle (1.7pt);
     \draw (\x,\y,\z) circle (1.7pt);
    }
   }
  }
  \foreach \p in {a,b,c,d,e,f,g,h,i,j}
    \filldraw (\p) circle (1.7pt);
  \refstepcounter{ris}
  \draw (1.5,1.5,-0.7) node {Figure \arabic{ris}.\label{figure: maximal basic set in 4x4x4}};
 \end{scope}
\end{tikzpicture}
\end{center}

\begin{remark}
  Slightly modifying Lemma~\ref{lemma: estimations}, we can prove that if a basic subset~$M$ belongs to a~parallelepiped of size $n_1 \times \ldots \times n_d$, then
  $$
    |M| \leqslant (n_1 + \ldots + n_d) - (d-1).
  $$
\end{remark}

\begin{proof}[Proof of Theorem~\ref{th: M is non-basic => 2n-1 < |M| < dn - (d-3)}]
  To obtain the upper bound, it is sufficient to use Theorem~\ref{th: M is basic => |M| < dn - (d-2)}. To obtain the lower bound, suppose that $|M|<2n$. This supposition immediately implies that there exists a layer $L$ such that $|L\cap M|<2$. Due to the conditions of the statement, $L\cap M\ne\emptyset$, hence, there is a single element $\x\in L\cap M$. As a consequence, if $M$ is non-basic, so is $M\setminus\{\x\}$. Therefore, $M$ is not minimal, which contradicts our supposition.
\end{proof}

\begin{remark} \label{Remark: lower bound is reachable}
  The lower bound of Theorem~\ref{th: M is non-basic => 2n-1 < |M| < dn - (d-3)} cannot be improved. To ensure, consider
  $$
   M = \{(k,k,\ldots,k), (k+1,k,\ldots,k) \mid k\in[n-1]\} \cup \{n,n,\ldots,n\} \cup \{1,n,\ldots,n\}.
  $$
  (see Fig.~\ref{figure: minimal non-basic set in 4x4x4}). We can see, that $|M|=2n$ and $M$ is non-basic, since it is annihilated by
  $$
   f = \sum\limits_{k=1}^{n-1} \Big(\textbf{1}_{(k,k,\ldots,k)} - \textbf{1}_{(k+1,k,\ldots,k)}\Big) + \Big(\textbf{1}_{(n,n,\ldots,n)} - \textbf{1}_{(1,n,\ldots,n)}\Big).
  $$

\begin{center}
\begin{tikzpicture}[scale=1.2,x={(1cm,0cm)},y={(0.3cm,0.2cm)},z={(0cm,1cm)},line width=.5pt]
 \begin{scope}
  \coordinate (a) at (0,0,0);
  \coordinate (b) at (1,0,0);
  \coordinate (c) at (1,1,1);
  \coordinate (d) at (2,1,1);
  \coordinate (e) at (2,2,2);
  \coordinate (f) at (3,2,2);
  \coordinate (g) at (3,3,3);
  \coordinate (h) at (0,3,3);
  \foreach \p in {0,1,2,3}{
    \draw [line width=.8pt] (0,\p,3)--++(3,0,0)--++(0,0,-3);
    \draw [line width=.8pt] (0,0,\p)--++(3,0,0)--++(0,3,0);
  }
  \foreach \p in {0,1,2}{
    \draw [line width=.8pt] (\p,0,0)--++(0,0,3)--++(0,3,0);
    \draw [dashed] (0,\p+1,3)--++(0,0,-3)--++(3,0,0);
    \draw [dashed] (\p,0,0)--++(0,3,0)--++(0,0,3);
  }
  \foreach \p in {1,2}{
   \draw [dashed] (0,0,\p)--++(0,3,0)--++(3,0,0);
   \foreach \q in {1,2}{
    \draw [dashed] (0,\p,\q)--++(3,0,0);
    \draw [dashed] (\p,0,\q)--++(0,3,0);
    \draw [dashed] (\p,\q,0)--++(0,0,3);
   }
  }
  \foreach \x in {0,1,2,3}{
   \foreach \y in {0,1,2,3}{
    \foreach \z in {0,1,2,3}{
     \filldraw[white] (\x,\y,\z) circle (1.7pt);
     \draw (\x,\y,\z) circle (1.7pt);
    }
   }
  }
  \foreach \p in {a,b,c,d,e,f,g,h}
    \filldraw (\p) circle (1.7pt);
  \refstepcounter{ris}
  \draw (1.5,1.5,-0.7) node {Figure \arabic{ris}.\label{figure: minimal non-basic set in 4x4x4}};
 \end{scope}
\end{tikzpicture}
\end{center}

 On the other hand, it is not clear whether it is possible to improve the upper bound or not. For instance, if we add an element to the basic subset   
  $$
   M = \{(k,1,\ldots,1), (1,k,\ldots,1), \ldots, (1,1,\ldots,k) \mid k\in[n]\}
  $$
  discussed in the proof of Theorem~\ref{th: M is basic => |M| < dn - (d-2)}, then we get a non-basic set which is not minimal. Say, if we add $\x=(x_1,\ldots,x_d)$, where $x_k>1$ for all $k\in[d]$, then the set $M\cup\{\x\}$ is annihilated by the function
  $$
   (d-1)\mathbf{1}_{(1,1,\ldots,1)} +
   \mathbf{1}_{\x} -
   \big(\mathbf{1}_{(x_1,1,\ldots,1)} +
    \mathbf{1}_{(1,x_2,\ldots,1)} +
    \ldots +
    \mathbf{1}_{(1,1,\ldots,x_d)}\big),
  $$
  and hence, $M\cup\{\x\}$ has a non-basic subset of size $d+2$.
\end{remark}

\section{Criterion for certain subsets to be basic} \label{Section:graphs_app}

The main goal of this section is to prove Theorem~\ref{th_set_to_graph} and Theorem~\ref{th_basis_graph}. The second part of the section assumes that the reader is familiar with the basics of hypergraph theory. For an extensive account of this topic, we refer, for example, to~\cite{Bretto2013}.

\begin{proof}[Proof of Theorem~\ref{th_set_to_graph}]
 Let us assume that there is an appropriate coloring of a~subset $K\subset M$. Then this coloring corresponds to the annihilation function $f\colon K\to\{\pm 1\}$. Extending $f$ to the set $M$ by defining $f(\x)=0$ for $\x\in M\setminus K$, we obtain an annihilation function of $M$. Hence, by Lemma~\ref{lemma: M is non-basic <=> annihilation weight function}, $M$~is non-basic.
 
 Conversely, let $M$ be non-basic. By Lemma~\ref{lemma: M is non-basic <=> annihilation weight function} this implies that there exists a~non-trivial annihilation function $f$ of $M$. Denote $K\subset M$ to be the domain of $f$ and define a function $g\colon M\to\{-1,0,1\}$ by
 \[
  g(\x) = \left\{
  \begin{array}{rl}
   f(\x)/|f(\x)|, & \mbox{if } \x\in K\\
   0, & \mbox{if } \x\in M\setminus K.
  \end{array}
  \right.
 \]
 Since every layer contains two or zero elements of $M$, the function $g$ is annihilation. Hence, it provides us a desired coloring.
\end{proof}

 A subset $M\subset[n]^d$ can be naturally considered as a~hypergraph whose vertices are elements of~$M$ and hyperedges are subsets of elements of $M$ that lie in the same layer. We denote this hypergraph~$G(M)$. If every layer contains two or zero elements of $M$, then the hypergraph becomes a~graph, possibly with multiple edges.

 The notion of finite basic subsets can be naturally extended to hypergraphs as follows.

 \begin{definition} \label{basis_graph}
  We call a hypergraph $G = (V, E)$ \emph{basic}, if for any vertex weight function $w_V \colon V \to \mathbb{R}$ there exists an edge weight function $w_E \colon E \to \mathbb{R}$ such that the weight of any vertex is equal to the sum of the weights of the edges incident to this vertex. In other words, for any $v \in V$:
  \begin{equation}\label{eq: basic graph condition}
   w_V(v) = \sum\limits_{e\colon v \in e} w_E(e).
  \end{equation}
 \end{definition}

 \begin{lemma} \label{lemma: set_to_hypergraph}
  A subset $M\subset[n]^d$ is basic if and only if the corresponding hypergraph~$G(M)$ is basic.
 \end{lemma}
 \begin{proof}
  It follows directly from the observation that there is a natural bijection between the function $f$ in Eq.~\eqref{eq: basic set condition} and the vertex weight function $w_V$ in Eq.~\eqref{eq: basic graph condition}. Thus, a collection of functions $f_1, \ldots, f_d$ determines the edge weight function $w_E$ and vice versa.
 \end{proof}

 \begin{definition} \label{co-boundary}
  Let $G = (V, E)$ be a hypergraph. The \emph{co-boundary} of a~vertex $v \in V$ is an edge weight function $\delta_v$ defined be the formula
  $$
   \delta_v(e) = \left\{\begin{array}{cc}
    1, & v \in e \\
    0, & v \notin e.
   \end{array}\right.
  $$
  In other words, the co-boundary $\delta_v$ is the indicator function of the subset of all edges incident to $v$. Also, the reader can interpret it as a row of the incidence matrix of $G$ corresponding to the vertex $v$.
 \end{definition}
 \begin{remark}
  The notion of co-boundary comes from \cite{Prasolov2006}, although there it has a~slightly different form. Note, that by the co-boundary of a vertex of a~hypergraph we also mean a particular case of the co-boundary of a vertex of a graph.
 \end{remark}

 \begin{lemma} \label{claim1} 
  A hypergraph $G = (V, E)$ is basic if and only if the co-boundaries of its vertices are linearly independent in $\mathbb{R}^{|E|}$.
 \end{lemma}
 \begin{proof}%[Proof of lemma \ref{claim1}]
 Let $V = \{v_1,\ldots,v_n\}$, $E = \{e_1,\ldots,e_m\}$ and $A$ be the incidence matrix of the hypergraph $G$. Then $G$ is basic if and only if for any vertex weight function $w_V$ there exists an edge weight function $w_E$ such that
 $$
  A\begin{pmatrix}
    w_E(e_1) \\
    \vdots  \\
    w_E(e_m)  \\
  \end{pmatrix} =
  \begin{pmatrix}
    w_V(v_1) \\
    \vdots  \\
    w_V(v_n)  \\
  \end{pmatrix},
 $$
 meaning, that rows of $A$ (i.e. co-boundaries) are linearly independent.
\end{proof}

\begin{lemma} \label{claim2} 
 The co-boundaries of vertices of a graph $G=(V,E)$ are linearly independent in $\mathbb{R}^{|E|}$ if and only if $G$ does not contain a bipartite connected component.
\end{lemma}
\begin{proof}%[Proof of lemma \ref{claim2}]
 Assume that there is a linear dependence
 \begin{equation}\label{eq: dependence}
  \sum \limits_{i = 1}^{n} \lambda_i \delta_{v_i} = 0.
 \end{equation}
 If there is an edge between vertices $v_i$ and $v_j$, then $\lambda_i = -\lambda_j$. Thus, vertices in every connected component are divided into two equivalence classes with coefficients $\lambda_i$ and $-\lambda_i$ respectively. Vertices from each of the classes are adjacent only to vertices from the other class. Hence, this component of the graph is bipartite.
 
 Conversely, if $G$ contains a bipartite component $C$ with parts $A$ and $B$, then there exists a linear dependence of form~\eqref{eq: dependence} with
  $$
   \lambda_i = \left\{\begin{array}{rl}
    1, & v_i \in A \\
    -1, & v_i \in B \\
    0, & v_i \notin C.
   \end{array}\right.
  $$
 \end{proof}

\begin{proof}[Proof of Theorem~\ref{th_basis_graph}]
  It is sufficient to apply Lemma~\ref{claim1} and Lemma~\ref{claim2}.
\end{proof}

\begin{corollary}\label{cor:graph_to_set}
 Let the intersection of a subset $M\subset[n]^d$ with each layer consist of two or zero points. Then $M$ is basic if and only if the corresponding graph $G(M)$ does not contain a bipartite connected component.
\end{corollary}

\begin{example} \label{example1}
 Let $M = \{(2,1,1), (1,2,1), (1,1,2), (2,2,2)\}$ (Fig.~\ref{figure: basic set, 4 points}). Then the corresponding graph $G(M)$ is the complete graph $K_4$ with four vertices. Since $K_4$ is connected and not bipartite, in accordance with Corollary~\ref{cor:graph_to_set} the set $M$ is basic.
\end{example}

\begin{center}
\begin{tikzpicture}[scale=1.2,x={(1cm,0cm)},y={(0.54cm,0.36cm)},z={(0cm,1cm)},line width=.5pt]
 \begin{scope}
  \coordinate (a1) at (0,0,0);
  \coordinate (a2) at (1,0,0);
  \coordinate (a3) at (1,1,0);
  \coordinate (a4) at (0,1,0);
  \coordinate (b1) at (0,0,1);
  \coordinate (b2) at (1,0,1);
  \coordinate (b3) at (1,1,1);
  \coordinate (b4) at (0,1,1);
  \draw [line width=.8pt] (b1)--(b4)--(b3)--(b2)--(b1)--(a1)--(a2)--(b2);
  \draw [line width=.8pt] (a2)--(a3)--(b3);
  \draw [dashed] (a1)--(a4);
  \draw [dashed] (b4)--(a4)--(a3);
  \foreach \p in {a1,a3,b2,b4}{
    \filldraw[white] (\p) circle (1.7pt);
    \draw (\p) circle (1.7pt);
  }
  \foreach \p in {a2,a4,b1,b3}
    \filldraw (\p) circle (1.7pt);
  \refstepcounter{ris}
  \draw (0.5,0.5,-0.6) node {Figure \arabic{ris}.\label{figure: basic set, 4 points}};
 \end{scope}
\end{tikzpicture}
\end{center}

\section{Structure of non-basic subsets}\label{Section:structure_of_non-basic_sets}

%The main goal of this section is to prove Theorem~\ref{th_irreducible_annihilation_function_is_unbounded}.

%\begin{definition}\label{defin: irreducible annihilation function}
% We call an annihilation function of a minimal non-basic set $M$ \emph{irreducible} (or \emph{in lowest terms}), if its values are setwise coprime integers.
%\end{definition}

  By Lemma~\ref{lemma: M is min. non-basic <=> annihilation weight function is unique}, if a non-basic subset $M \subset [n]^d$ is minimal, then there exists the unique annihilation function up to multiplying by a constant. Since the matrix~$A_M$ has integer entries, it is possible to choose the annihilation function $f$ with integer values that are setwise coprime integers, so that the notion of irreducible annihilation function is well-defined. The name \emph{irreducible} is chosen by analogy with integer fractions.

\begin{example} \label{example: non-basic 5 points in 2x2x2}
  If, as it is shown in Fig.~\ref{figure: non-basic set, 5 points},
  $$
    M = \{(1,1,1), (2,1,1), (1,2,1), (1,1,2), (2,2,2)\},
  $$
  then the irreducible annihilation function of $M$ has the following form:
  $$
   f = 2\cdot\textbf{1}_{(1,1,1)} -
     		 \textbf{1}_{(2,1,1)} -
       		 \textbf{1}_{(1,2,1)} -
       		 \textbf{1}_{(1,1,2)} +
       		 \textbf{1}_{(2,2,2)}
  $$
\end{example}

\begin{center}
\begin{tikzpicture}[scale=1.2,x={(1cm,0cm)},y={(0.54cm,0.36cm)},z={(0cm,1cm)},line width=.5pt]
 \begin{scope}
  \coordinate (a1) at (0,0,0);
  \coordinate (a2) at (1,0,0);
  \coordinate (a3) at (1,1,0);
  \coordinate (a4) at (0,1,0);
  \coordinate (b1) at (0,0,1);
  \coordinate (b2) at (1,0,1);
  \coordinate (b3) at (1,1,1);
  \coordinate (b4) at (0,1,1);
  \draw [line width=.8pt] (b1)--(b4)--(b3)--(b2)--(b1)--(a1)--(a2)--(b2);
  \draw [line width=.8pt] (a2)--(a3)--(b3);
  \draw [dashed] (a1)--(a4);
  \draw [dashed] (b4)--(a4)--(a3);
  \foreach \p in {a3,b2,b4}{
    \filldraw[white] (\p) circle (1.7pt);
    \draw (\p) circle (1.7pt);
  }
  \foreach \p in {a1,a2,a4,b1,b3}
    \filldraw (\p) circle (1.7pt);
  \refstepcounter{ris}
  \draw (0.5,0.5,-0.6) node {Figure \arabic{ris}.\label{figure: non-basic set, 5 points}};
 \end{scope}
\end{tikzpicture}
\end{center}

\begin{proof}[Proof of Theorem~\ref{th_irreducible_annihilation_function_is_unbounded}]

  Let $a_1,\ldots,a_m,b_1,\ldots,b_m,c_1,\ldots,c_m,d,e$ be different integers. Define the set $M$ to consist of the following $6m+2$ elements (see Fig.~\ref{figure: non-basic set, 6k+2 points}):
  \begin{itemize}
    \item $3m$ points with coordinates $(d,a_k,a_k)$, $(b_k,d,b_k)$, $(c_k,c_k,d)$, $k \in [m]$;
    \item $3m$ points with coordinates $(e,a_k,a_{k+1})$, $(b_k,e,b_{k+1})$, $(c_k,c_{k+1},e)$,\\ $k \in [m]$ (here, we assume that $a_{m+1}=a_1$, $b_{m+1}=b_1$ and $c_{m+1}=c_1$);
    \item $1$ point with coordinates $(d,d,d)$;
    \item $1$ point with coordinates $(e,e,e)$.
  \end{itemize}

\begin{center}
\begin{tikzpicture}[line width=.8pt]
 \begin{scope}
  \filldraw[rounded corners, green, opacity = 0.2] (6.0,2.3) -- (-0.3,1.2) .. controls (-0.5,0.5) and (3.5,0.5) .. (3.3,1.2)--(6.0,2)--cycle;
  \filldraw[rounded corners, yellow, opacity = 0.4] (5.5,2.3) -- (3.7,1.2) .. controls (3.5,0.5) and (7.5,0.5) .. (7.3,1.2)--cycle;
  \filldraw[rounded corners, red, opacity = 0.2] (5.0,2.3) -- (11.3,1.2) .. controls (11.5,0.5) and (7.5,0.5) .. (7.7,1.2)--(5.0,2)--cycle;
  \draw[rounded corners] (6.0,2.3) -- (-0.3,1.2) .. controls (-0.5,0.5) and (3.5,0.5) .. (3.3,1.2)--(6.0,2)--cycle;
  \draw[rounded corners] (5.5,2.3) -- (3.7,1.2) .. controls (3.5,0.5) and (7.5,0.5) .. (7.3,1.2)--cycle;
  \draw[rounded corners] (5.0,2.3) -- (11.3,1.2) .. controls (11.5,0.5) and (7.5,0.5) .. (7.7,1.2)--(5.0,2)--cycle;
  \filldraw[rounded corners, red, opacity = 0.2] (6.0,-1.3) -- (-0.3,-0.2) .. controls (-0.5,0.5) and (3.5,0.5) .. (3.3,-0.2)--(6.0,-1)--cycle;
  \filldraw[rounded corners, yellow, opacity = 0.4] (5.5,-1.3) -- (3.7,-0.2) .. controls (3.5,0.5) and (7.5,0.5) .. (7.3,-0.2)--cycle;
  \filldraw[rounded corners, green, opacity = 0.2] (5.0,-1.3) -- (11.3,-0.2) .. controls (11.5,0.5) and (7.5,0.5) .. (7.7,-0.2)--(5.0,-1)--cycle;
  \draw[rounded corners] (6.0,-1.3) -- (-0.3,-0.2) .. controls (-0.5,0.5) and (3.5,0.5) .. (3.3,-0.2)--(6.0,-1)--cycle;
  \draw[rounded corners] (5.5,-1.3) -- (3.7,-0.2) .. controls (3.5,0.5) and (7.5,0.5) .. (7.3,-0.2)--cycle;
  \draw[rounded corners] (5.0,-1.3) -- (11.3,-0.2) .. controls (11.5,0.5) and (7.5,0.5) .. (7.7,-0.2)--(5.0,-1)--cycle;
  \foreach \x in {0,...,2}{
   \foreach \y in {0,...,2}{
    \draw (4*\x+\y,0)--++(1,1);
   }
   \draw (4*\x,1)--++(3,-1);
  }
  \foreach \x in {0,...,11}{
    \draw (\x,0)--++(0,1);
    \filldraw (\x,0) circle (1.7pt);
    \filldraw (\x,1) circle (1.7pt);
  }
  \filldraw (5.5,2) circle (1.7pt);
  \filldraw (5.5,-1) circle (1.7pt);
  \refstepcounter{ris}
  \draw (5.5,-1.8) node {Figure \arabic{ris}.\label{figure: non-basic set, 6k+2 points}};
 \end{scope}
\end{tikzpicture}
\end{center}
  Then the function $f$ taking the values $1$, $-1$, $-m$ and $m$ at the points of first, second, third and fourth groups respectively is irreducible annihilation.
  
  Now we show that the set $M$ is minimal. Indeed, $2m$ points with coordinates $(d,a_k,a_k)$ and $(e,a_k,a_{k+1})$ constitute a cycle of order $2m$, whose elements belong (or do not belong) to a minimal subset $N$ of $M$ simultaneously. Two other cycles behave the same way. To finish the proof, we need to mention that points $(d,d,d)$ and $(e,e,e)$ belong to $N$ for sure, and if $(e,e,e)$ belongs to $N$, then each of three cycles does as well. Hence, all $6m+2$ points are in~$N$ and $N=M$.
\end{proof}

\begin{definition}\label{def: simple annihilation function}
  Let $\p,\q,\r,\s$ be the consecutive vertices of a rectangle whose sides are parallel to the coordinate axes. A \emph{simple annihilation function} is
  $$
   f_{\p\q\r\s} = \textbf{1}_{\p} -  \textbf{1}_{\q} + \textbf{1}_{\r} - \textbf{1}_{\s}.
  $$
\end{definition}

  By double induction on $d$ and $n$, one can verify that every annihilation function of $[n]^d$ can be decomposed into a~finite sum of simple annihilation functions (Theorem~\ref{thBoyarov}). The proof is routine and will be omitted. 

\begin{example} \label{example: non-basic 5 points in 2x2x2 revisited}
  For the set $M$ from Example~\ref{example: non-basic 5 points in 2x2x2}, its annihilation function
  $$
   f = 2\cdot\textbf{1}_{(1,1,1)} -
     		 \textbf{1}_{(2,1,1)} -
       		 \textbf{1}_{(1,2,1)} -
       		 \textbf{1}_{(1,1,2)} +
       		 \textbf{1}_{(2,2,2)}
  $$
  is decomposed as the sum of simple annihilation functions
  $$
   \textbf{1}_{(1,1,1)} -
   \textbf{1}_{(1,2,1)} +
   \textbf{1}_{(2,2,1)} -
   \textbf{1}_{(2,1,1)},
  $$
  $$
   \textbf{1}_{(1,1,1)} -
   \textbf{1}_{(1,1,2)} +
   \textbf{1}_{(2,1,2)} -
   \textbf{1}_{(2,1,1)}
  $$
  and  
  $$
   \textbf{1}_{(2,1,1)} -
   \textbf{1}_{(2,2,1)} +
   \textbf{1}_{(2,2,2)} -
   \textbf{1}_{(2,1,2)}.
  $$
\end{example}

\section{Conclusion} \label{Section:conclusion}

  As we have seen in Section~\ref{Section:estimations}, Theorem~\ref{th: M is non-basic => 2n-1 < |M| < dn - (d-3)} provides some bounds for the number of elements in certain non-basic subsets of $[n]^d$. While the lower bound is reachable (see Remark~\ref{Remark: lower bound is reachable}), it is not clear, whether the upper bound and the intermediate values are reachable as well. This observation leads us to the following question.

\begin{question}\label{question: reachability}
  Are the intermediate values in Theorem~\ref{th: M is non-basic => 2n-1 < |M| < dn - (d-3)} reachable? In other words, is it true that for any integer $k$, $2n < k \leqslant dn - (d-2)$, there exists a~minimal non-basic subset $M\subset[n]^d$ of size $k$ such that every layer of~$[n]^d$ has a non-empty intersection with $M$?
\end{question}

  Figures~\ref{figure: non-basic sets in 3x3x3},~\ref{figure: non-basic sets, n=4, |M|=8,9} and~\ref{figure: non-basic sets, n=4, |M|=10,11} shows that for $d=3$, $n\in\{3,4\}$ the answer to Question~\ref{question: reachability} is positive (here, a point color indicates the value of the irreducible annihilation function: red, blue, green and black are reserved for $-1$, $1$, $-2$ and $2$ respectively).
\begin{center}
\begin{tikzpicture}[scale=1.2,x={(1cm,0cm)},y={(0.4cm,0.3cm)},z={(0cm,1cm)},line width=.5pt]
 \begin{scope}
  \coordinate (a) at (2,0,0);
  \coordinate (b) at (2,1,0);
  \coordinate (c) at (1,0,1);
  \coordinate (d) at (1,2,1);
  \coordinate (e) at (0,1,2);
  \coordinate (f) at (0,2,2);
  \foreach \p in {0,1,2}{
    \draw [line width=.8pt] (0,\p,2)--++(2,0,0)--++(0,0,-2);
    \draw [line width=.8pt] (0,0,\p)--++(2,0,0)--++(0,2,0);
  }
  \foreach \p in {0,1}{
    \draw [line width=.8pt] (\p,0,0)--++(0,0,2)--++(0,2,0);
    \draw [dashed] (0,\p+1,2)--++(0,0,-2)--++(2,0,0);
    \draw [dashed] (\p,0,0)--++(0,2,0)--++(0,0,2);
  }
  \draw [dashed] (0,0,1)--++(0,2,0)--++(2,0,0);
  \draw [dashed] (0,1,1)--++(2,0,0);
  \draw [dashed] (1,0,1)--++(0,2,0);
  \draw [dashed] (1,1,0)--++(0,0,2);
  \foreach \x in {0,1,2}{
   \foreach \y in {0,1,2}{
    \foreach \z in {0,1,2}{
     \filldraw[white] (\x,\y,\z) circle (1.7pt);
     \draw (\x,\y,\z) circle (1.7pt);
    }
   }
  }
  \foreach \p in {a,d,e}{
    \filldraw[red] (\p) circle (1.7pt);
    \draw (\p) circle (1.7pt);
  }
  \foreach \p in {b,c,f}{
    \filldraw[blue] (\p) circle (1.7pt);
    \draw (\p) circle (1.7pt);
  }
 \end{scope}
 \begin{scope}[xshift=4cm]
  \coordinate (a) at (2,0,0);
  \coordinate (b) at (2,2,1);
  \coordinate (c) at (1,1,1);
  \coordinate (d) at (1,2,2);
  \coordinate (e) at (0,0,0);
  \coordinate (f) at (0,1,2);
  \coordinate (g) at (0,2,2);
  \foreach \p in {0,1,2}{
    \draw [line width=.8pt] (0,\p,2)--++(2,0,0)--++(0,0,-2);
    \draw [line width=.8pt] (0,0,\p)--++(2,0,0)--++(0,2,0);
  }
  \foreach \p in {0,1}{
    \draw [line width=.8pt] (\p,0,0)--++(0,0,2)--++(0,2,0);
    \draw [dashed] (0,\p+1,2)--++(0,0,-2)--++(2,0,0);
    \draw [dashed] (\p,0,0)--++(0,2,0)--++(0,0,2);
  }
  \draw [dashed] (0,0,1)--++(0,2,0)--++(2,0,0);
  \draw [dashed] (0,1,1)--++(2,0,0);
  \draw [dashed] (1,0,1)--++(0,2,0);
  \draw [dashed] (1,1,0)--++(0,0,2);
  \foreach \x in {0,1,2}{
   \foreach \y in {0,1,2}{
    \foreach \z in {0,1,2}{
     \filldraw[white] (\x,\y,\z) circle (1.7pt);
     \draw (\x,\y,\z) circle (1.7pt);
    }
   }
  }
  \foreach \p in {g}
    \filldraw (\p) circle (1.7pt);
  \foreach \p in {b,d,e,f}{
    \filldraw[red] (\p) circle (1.7pt);
    \draw (\p) circle (1.7pt);
  }
  \foreach \p in {a,c}{
    \filldraw[blue] (\p) circle (1.7pt);
    \draw (\p) circle (1.7pt);
  }
  \refstepcounter{ris}
  \draw (1,1,-0.7) node {Figure \arabic{ris}.\label{figure: non-basic sets in 3x3x3}};
 \end{scope}
 \begin{scope}[xshift=8cm]
  \coordinate (a) at (2,0,0);
  \coordinate (b) at (2,0,1);
  \coordinate (c) at (2,1,0);
  \coordinate (d) at (1,0,0);
  \coordinate (e) at (1,2,2);
  \coordinate (f) at (0,1,2);
  \coordinate (g) at (0,2,1);
  \coordinate (h) at (0,2,2);
  \foreach \p in {0,1,2}{
    \draw [line width=.8pt] (0,\p,2)--++(2,0,0)--++(0,0,-2);
    \draw [line width=.8pt] (0,0,\p)--++(2,0,0)--++(0,2,0);
  }
  \foreach \p in {0,1}{
    \draw [line width=.8pt] (\p,0,0)--++(0,0,2)--++(0,2,0);
    \draw [dashed] (0,\p+1,2)--++(0,0,-2)--++(2,0,0);
    \draw [dashed] (\p,0,0)--++(0,2,0)--++(0,0,2);
  }
  \draw [dashed] (0,0,1)--++(0,2,0)--++(2,0,0);
  \draw [dashed] (0,1,1)--++(2,0,0);
  \draw [dashed] (1,0,1)--++(0,2,0);
  \draw [dashed] (1,1,0)--++(0,0,2);
  \foreach \x in {0,1,2}{
   \foreach \y in {0,1,2}{
    \foreach \z in {0,1,2}{
     \filldraw[white] (\x,\y,\z) circle (1.7pt);
     \draw (\x,\y,\z) circle (1.7pt);
    }
   }
  }
  \foreach \p in {h}
    \filldraw (\p) circle (1.7pt);
  \foreach \p in {a}{
    \filldraw[green] (\p) circle (1.7pt);
    \draw (\p) circle (1.7pt);
  }
  \foreach \p in {e,f,g}{
    \filldraw[red] (\p) circle (1.7pt);
    \draw (\p) circle (1.7pt);
  }
  \foreach \p in {b,c,d}{
    \filldraw[blue] (\p) circle (1.7pt);
    \draw (\p) circle (1.7pt);
  }
 \end{scope}
\end{tikzpicture}
\end{center}

\begin{center}
\begin{tikzpicture}[scale=1.2,x={(1cm,0cm)},y={(0.3cm,0.2cm)},z={(0cm,1cm)},line width=.5pt]
 \begin{scope}
  \coordinate (a) at (0,0,0);
  \coordinate (b) at (1,0,0);
  \coordinate (c) at (1,1,1);
  \coordinate (d) at (2,1,1);
  \coordinate (e) at (2,2,2);
  \coordinate (f) at (3,2,2);
  \coordinate (g) at (3,3,3);
  \coordinate (h) at (0,3,3);
  \foreach \p in {0,1,2,3}{
    \draw [line width=.8pt] (0,\p,3)--++(3,0,0)--++(0,0,-3);
    \draw [line width=.8pt] (0,0,\p)--++(3,0,0)--++(0,3,0);
  }
  \foreach \p in {0,1,2}{
    \draw [line width=.8pt] (\p,0,0)--++(0,0,3)--++(0,3,0);
    \draw [dashed] (0,\p+1,3)--++(0,0,-3)--++(3,0,0);
    \draw [dashed] (\p,0,0)--++(0,3,0)--++(0,0,3);
  }
  \foreach \p in {1,2}{
   \draw [dashed] (0,0,\p)--++(0,3,0)--++(3,0,0);
   \foreach \q in {1,2}{
    \draw [dashed] (0,\p,\q)--++(3,0,0);
    \draw [dashed] (\p,0,\q)--++(0,3,0);
    \draw [dashed] (\p,\q,0)--++(0,0,3);
   }
  }
  \foreach \x in {0,1,2,3}{
   \foreach \y in {0,1,2,3}{
    \foreach \z in {0,1,2,3}{
     \filldraw[white] (\x,\y,\z) circle (1.7pt);
     \draw (\x,\y,\z) circle (1.7pt);
    }
   }
  }
  \foreach \p in {a,c,e,g}{
    \filldraw[red] (\p) circle (1.7pt);
    \draw (\p) circle (1.7pt);
  }
  \foreach \p in {b,d,f,h}{
    \filldraw[blue] (\p) circle (1.7pt);
    \draw (\p) circle (1.7pt);
  }
 \end{scope}
 \begin{scope}[xshift=5cm]
  \coordinate (a) at (0,3,3);
  \coordinate (b) at (0,3,2);
  \coordinate (c) at (0,2,3);
  \coordinate (d) at (1,3,3);
  \coordinate (e) at (1,0,1);
  \coordinate (f) at (2,2,0);
  \coordinate (g) at (3,1,2);
  \coordinate (h) at (2,1,1);
  \coordinate (i) at (3,0,0);
  \foreach \p in {0,1,2,3}{
    \draw [line width=.8pt] (0,\p,3)--++(3,0,0)--++(0,0,-3);
    \draw [line width=.8pt] (0,0,\p)--++(3,0,0)--++(0,3,0);
  }
  \foreach \p in {0,1,2}{
    \draw [line width=.8pt] (\p,0,0)--++(0,0,3)--++(0,3,0);
    \draw [dashed] (0,\p+1,3)--++(0,0,-3)--++(3,0,0);
    \draw [dashed] (\p,0,0)--++(0,3,0)--++(0,0,3);
  }
  \foreach \p in {1,2}{
   \draw [dashed] (0,0,\p)--++(0,3,0)--++(3,0,0);
   \foreach \q in {1,2}{
    \draw [dashed] (0,\p,\q)--++(3,0,0);
    \draw [dashed] (\p,0,\q)--++(0,3,0);
    \draw [dashed] (\p,\q,0)--++(0,0,3);
   }
  }
  \foreach \x in {0,1,2,3}{
   \foreach \y in {0,1,2,3}{
    \foreach \z in {0,1,2,3}{
     \filldraw[white] (\x,\y,\z) circle (1.7pt);
     \draw (\x,\y,\z) circle (1.7pt);
    }
   }
  }
  \foreach \p in {a}
    \filldraw (\p) circle (1.7pt);
  \foreach \p in {b,c,d,h,i}{
    \filldraw[red] (\p) circle (1.7pt);
    \draw (\p) circle (1.7pt);
  }
  \foreach \p in {e,f,g}{
    \filldraw[blue] (\p) circle (1.7pt);
    \draw (\p) circle (1.7pt);
  }
 \end{scope}
 \refstepcounter{ris}\label{figure: non-basic sets, n=4, |M|=8,9}
 \draw (5.0,-1.8) node {Figure \arabic{ris}.};
\end{tikzpicture}
\end{center}

\begin{center}
\begin{tikzpicture}[scale=1.2,x={(1cm,0cm)},y={(0.3cm,0.2cm)},z={(0cm,1cm)},line width=.5pt]
 \begin{scope}
  \coordinate (a) at (0,3,3);
  \coordinate (b) at (3,0,0);
  \coordinate (c) at (3,3,2);
  \coordinate (d) at (2,3,0);
  \coordinate (e) at (3,1,3);
  \coordinate (f) at (1,0,3);
  \coordinate (g) at (0,2,0);
  \coordinate (h) at (0,0,1);
  \coordinate (i) at (2,1,1);
  \coordinate (j) at (1,2,2);
  \foreach \p in {0,1,2,3}{
    \draw [line width=.8pt] (0,\p,3)--++(3,0,0)--++(0,0,-3);
    \draw [line width=.8pt] (0,0,\p)--++(3,0,0)--++(0,3,0);
  }
  \foreach \p in {0,1,2}{
    \draw [line width=.8pt] (\p,0,0)--++(0,0,3)--++(0,3,0);
    \draw [dashed] (0,\p+1,3)--++(0,0,-3)--++(3,0,0);
    \draw [dashed] (\p,0,0)--++(0,3,0)--++(0,0,3);
  }
  \foreach \p in {1,2}{
   \draw [dashed] (0,0,\p)--++(0,3,0)--++(3,0,0);
   \foreach \q in {1,2}{
    \draw [dashed] (0,\p,\q)--++(3,0,0);
    \draw [dashed] (\p,0,\q)--++(0,3,0);
    \draw [dashed] (\p,\q,0)--++(0,0,3);
   }
  }
  \foreach \x in {0,1,2,3}{
   \foreach \y in {0,1,2,3}{
    \foreach \z in {0,1,2,3}{
     \filldraw[white] (\x,\y,\z) circle (1.7pt);
     \draw (\x,\y,\z) circle (1.7pt);
    }
   }
  }
  \foreach \p in {a,b}
    \filldraw (\p) circle (1.7pt);
  \foreach \p in {c,d,e,f,g,h}{
    \filldraw[red] (\p) circle (1.7pt);
     \draw (\p) circle (1.7pt);
  }
  \foreach \p in {i,j}{
    \filldraw[blue] (\p) circle (1.7pt);
     \draw (\p) circle (1.7pt);
  }
 \end{scope}
 \begin{scope}[xshift=5cm]
  \coordinate (a) at (0,3,3);
  \coordinate (b) at (3,0,0);
  \coordinate (c) at (1,2,3);
  \coordinate (d) at (1,2,2);
  \coordinate (e) at (1,1,1);
  \coordinate (f) at (2,2,1);
  \coordinate (g) at (2,3,0);
  \coordinate (h) at (3,0,2);
  \coordinate (i) at (0,0,1);
  \coordinate (j) at (3,3,1);
  \coordinate (k) at (0,1,0);
  \coordinate (l) at (2,3,0);
  \foreach \p in {0,1,2,3}{
    \draw [line width=.8pt] (0,\p,3)--++(3,0,0)--++(0,0,-3);
    \draw [line width=.8pt] (0,0,\p)--++(3,0,0)--++(0,3,0);
  }
  \foreach \p in {0,1,2}{
    \draw [line width=.8pt] (\p,0,0)--++(0,0,3)--++(0,3,0);
    \draw [dashed] (0,\p+1,3)--++(0,0,-3)--++(3,0,0);
    \draw [dashed] (\p,0,0)--++(0,3,0)--++(0,0,3);
  }
  \foreach \p in {1,2}{
   \draw [dashed] (0,0,\p)--++(0,3,0)--++(3,0,0);
   \foreach \q in {1,2}{
    \draw [dashed] (0,\p,\q)--++(3,0,0);
    \draw [dashed] (\p,0,\q)--++(0,3,0);
    \draw [dashed] (\p,\q,0)--++(0,0,3);
   }
  }
  \foreach \x in {0,1,2,3}{
   \foreach \y in {0,1,2,3}{
    \foreach \z in {0,1,2,3}{
     \filldraw[white] (\x,\y,\z) circle (1.7pt);
     \draw (\x,\y,\z) circle (1.7pt);
    }
   }
  }
  \foreach \p in {a,b}{
    \filldraw[black] (\p) circle (1.7pt);
     \draw (\p) circle (1.7pt);
  }
  \foreach \p in {c}{
    \filldraw[green] (\p) circle (1.7pt);
     \draw (\p) circle (1.7pt);
  }
  \foreach \p in {g,h,i,j,k}{
    \filldraw[red] (\p) circle (1.7pt);
     \draw (\p) circle (1.7pt);
  }
  \foreach \p in {d,e,f}{
    \filldraw[blue] (\p) circle (1.7pt);
     \draw (\p) circle (1.7pt);
  }
 \end{scope}
  \refstepcounter{ris}
  \draw (5.0,-1.8) node {Figure \arabic{ris}.\label{figure: non-basic sets, n=4, |M|=10,11}};
\end{tikzpicture}
\end{center}

 Surprisingly, in all known examples, including the sets shown in Figs.~\ref{figure: minimal non-basic set in 4x4x4}--\ref{figure: non-basic sets, n=4, |M|=10,11}, the values of irreducible annihilation functions present a specific behavior. This allows us to state the following conjecture.

\begin{conjecture}
  If $M \subset [n]^3$ is a minimal non-basic subset such that every layer of $[n]^3$ has a non-empty intersection with $M$, then its irreducible annihilation function $f$ satisfies
  $$
    \sum\limits_{\x\in M} |f(\x)| = 2\big(|M| - n\big).
  $$
\end{conjecture}

 As we mentioned before, due to Theorem~\ref{th_irreducible_annihilation_function_is_unbounded}, there is no reason to expect the existence of simple criterion (similar to Theorem~\ref{th_set_to_graph}) for a subset of $[n]^d$ to be basic in the general case $d\geqslant3$. Still, it does not mean that there is no simplification at all, and it would be interesting to find one, at least for $d=3$ or for the case of small values of~$n$. 

 Another possible direction for research would come from the generalization of the initial problem to hypergraphs. As we have seen in Section~\ref{Section:graphs_app}, the concept of basic hypergraphs admits almost the same interpretation in algebraic terms as the one of basic subsets. We can make this similarity even deeper as follows. Let $G = (V, E)$ be a~hypergraph. Define a linear map $\Psi \colon \mathbb{R}^{|V|} \to \mathbb{R}^{|E|}$ on indicators,
  $$
    \Psi(\textbf{1}_{v}) = \sum\limits_{e\colon v \in e} \textbf{1}_e = \delta_v,
  $$
 and extend it on $\mathbb{R}^{|V|}$ by linearity. % (compare with the proof of Theorem~\ref{th: M is basic => |M| < dn - (d-2)}, Eq.~\eqref{eq: definition of Psi}).
 In other words, for any $v \in V$, we set the image of the indicator function~$\textbf{1}_{v}$ to be the co-boundary $\delta_v$. Then the analogue of Lemma~\ref{lemma: M is non-basic <=> annihilation weight function} is that the hypergraph is basic if and only if $\ker\Psi$ is trivial, while the analogue of Lemma~\ref{lemma: M is min. non-basic <=> annihilation weight function is unique} is that for minimal non-basic hypergraphs we have $\dim\ker\Psi=1$. At the same time, the analogue of Theorem~\ref{th: M is basic => |M| < dn - (d-2)} is that for a basic hypergraph $G = (V, E)$, on has $|V| \leqslant |E|$ (which is trivial).

 It is natural to state the following general question.

\begin{question}\label{question: basic hypergraph}
  What are the conditions for a hypergraph to be basic?
\end{question}

For now, this question in its generality is open.

\section{Acknowledgements} \label{Section:acknowledgements}

 We thank A.B.~Skopenkov for useful discussions and criticism, N.~Volkov for searching examples of non-basic subsets and I.~Boyarov for proving the weaker version of Theorem~\ref{thBoyarov}.

 This work was partly supported by Russian Science Foundation Grant N~22-11-00177.

\end{document}